\documentclass{siamart}
\usepackage{amsmath, amssymb}
\usepackage{color}
\usepackage{graphicx,epsfig,epstopdf}

\newcommand{\lmin}{\lambda_{\min}}
\newcommand{\lmax}{\lambda_{\max}}
\newcommand{\spec}{\mathrm{spec}}

\newcommand{\mathC}{\mathbb{C}}
\newcommand{\mathR}{\mathbb{R}}

\def\cvd{~\vbox{\hrule\hbox{%
   \vrule height1.3ex\hskip0.8ex\vrule}\hrule } }

\title{Asynchronous Richardson iterations:\\Theory and practice\thanks{This version dated \today.
}}

\author{Edmond Chow\thanks{Georgia Institute of Technology, Atlanta, GA,
USA (\email{echow@cc.gatech.edu}).}
\and
Andreas Frommer\thanks{Bergische Universit\"{a}t Wuppertal,
Wuppertal, Germany (\email{frommer@math.uni-wuppertal.de}).}
\and
Daniel B. Szyld\thanks{Temple University, Philadelphia, PA, USA
(\email{szyld@temple.edu}).}
}

\begin{document}
\maketitle

\begin{abstract}
We consider asynchronous versions of the first and second order Richardson
methods for solving linear systems of equations.
These methods depend on parameters whose values are chosen {\em a priori}.  
We explore the parameter values that can be proven to give convergence
of the asynchronous methods.  This is the first such analysis for
asynchronous second order methods.  We find that for the first order
method, the optimal parameter value for the synchronous case also gives an
asynchronously convergent method. For the second order method, the parameter
ranges for which we can prove asynchronous convergence do not contain
the optimal parameter values for the synchronous iteration.  In practice,
however, the asynchronous second order iterations may still converge
using the optimal parameter values, 
or parameter values close to the optimal ones,
despite this result.  We explore this behavior with a multithreaded
parallel implementation of the asynchronous methods.
\end{abstract}



\section{Introduction}
\label{intro:sec}

A parallel asynchronous iterative method for solving a system of equations
is a fixed-point iteration in which processors do not synchronize at the
end of each iteration.  Instead, processors proceed iterating with the
latest data that is available from other processors.  Running an iterative
method in such an asynchronous fashion may reduce solution time when
there is an imbalance of the effective load between the processors because
fast processors do not need to wait for slow processors.  Solution time
may also be reduced when interprocessor communication costs are high
because computation continues while communication takes place.
However, the convergence properties of a synchronous iterative method are
changed when running the method asynchronously.

Consider the system of equations $x=G(x)$ {in fixed point form}, where $G: \mathR^n \to \mathR^n$,
which can be written componentwise
as $x_i = g_i(x)$, $i=1,\ldots,n$.  An asynchronous
iterative method for solving this system of equations can be
defined mathematically as the sequence of updates
\cite{baudet-1978,bertsekas-tsitsiklis-book,chazan-miranker},
\[
x_i^{k} =
\begin{cases}
x_i^{k-1} & \mbox{if\ } i \notin J_k \\
g_i(x_1^{s_1(k)}, 
    x_2^{s_2(k)}, \ldots, 
    x_n^{s_n(k)}) & \mbox{if\ } i \in J_k
\end{cases},
\]
where $x_i^{k}$ denotes component $i$ of the iterate 
at time instant $k$, $J_k$ is the set of
indices updated at instant $k$, and $s_j(k) \le k-1$ is the 
last instant {component $j$} 
was updated before being read when evaluating
$g_i$ at instant $k$.  We point out that (a) not all
updates are performed at the same time instant, and (b) updates may use stale
information, which models communication delays in reading or writing.

With some natural assumptions on the sequence of updates above, much work
has been done on showing the conditions under which asynchronous iterative
methods converge; see the survey \cite{Frommer.Szyld.00}.
For linear systems, where $G(x) = Tx +c$, $T \in \mathR^{n \times n}$, $c \in \mathR^n$,
the pioneering result from \cite{chazan-miranker} states that, under very mild conditions
on the sets $J_k$ and the sequences $s_j(k)$, any
asynchronous iteration
converges for any initial vector if and only if $\rho(|T|)<1$. Here,
$|T| \in \mathR^{n \times n}$ arises from $T$ by taking absolute values for each entry and $\rho$ denotes the spectral radius. 
The mild conditions on $J_k$ and $s_j(k)$ are that 
\begin{align} 
& \lim_{k \to\infty} s_j(k) = \infty \mbox{ for } j=1,\ldots,n \enspace \mbox{ and } \label{eq:asy_delays}\\
& \mbox{each } i \in \{1,\ldots,n\} \mbox{ appears infinitely many times in the sets $J_k$} \label{eq:asy_updates} .
\end{align}
Since $\rho(T) \le \rho(|T|)$, it
appears that the condition for convergence of asynchronous iterations
is more strict than that of synchronous iterations.

For linear systems $Ax = b$, asynchronous iterative methods that are based on the
Jacobi or block Jacobi splitting, i.e., $T = I-D^{-1}A$ with $D$ the diagonal or block diagonal of $A$,
have been extensively studied, although
these splittings generally give slow convergence; see  
\cite{bethune1,hook1,jordi-ipdps18,wolfsonpou-chow-jpdc-2018} for some recent references.  In this paper, we
consider first and second order Richardson methods \cite{Richardson.10}.
If information on the bounds of the spectrum of $A$ is available, this can be used to determine the parameter values to use
for the Richardson methods, and
the second order Richardson method, in particular, then converges rapidly.  This paper
explores the parameter values that can be proven to give convergence
of asynchronous Richardson methods. In particular, it presents the 
first such analysis for second order methods. 

Statements about the {\em rate of convergence}, however, cannot be made without
a description of the sets $J_k$ and {the sequences} $s_j(k)$.  Both 
depend on
properties of the parallel computation, including how the problem is
partitioned among the processors, and computer characteristics such as
computation speed and
interprocessor communication latency and bandwidth.  Indeed, one can
imagine that in an asynchronous computation where communication is fast
and the workload is balanced, the asynchronous computation may behave very
much like the synchronous computation, while it may behave very differently if load is unbalanced or communication costs are high.
In this paper, we will therefore not go 
into the details of an analysis of the convergence rate, but we will demonstrate
the actual behavior of asynchronous first and second order Richardson
methods using a parallel multithreaded implementation of the methods.

Our theoretical and experimental
results are suggestive for an asynchronous version of the Chebyshev
semi-iterative method.  The Chebyshev method can be regarded as the
non-stationary counterpart of the stationary method which is the second
order Richardson method.  If one uses the optimal parameter values in
second order Richardson, i.e., the parameter values that minimize the
spectral radius of the iteration operator, then, asymptotically, both
second order Richardson and Chebychev iterations have the same convergence
rate \cite{Golub.Varga.61.partI}.  For a short historical description
of the development of these methods, see \cite{Saad.history.20}.
Unlike {those} Krylov subspace methods which rely on a variational principle,
the second order Richardson and Chebyshev
methods do not require inner products, which is what allows them to be easily executed 
asynchronously.

In recent related work, asynchronous versions of Schwarz and optimized
Schwarz methods have been developed \cite{Glusa.etal.20,Magoules:2017,yamazaki-2019}.


\section{The setting}

From the beginning, we assume that the original system  
\[
\hat{A} x = \hat{b}, \enspace \hat{A} \in \mathC^{n \times n}, \enspace \hat{b} \in \mathC^n
\]
is preconditioned with a nonsingular matrix $M$,
that is, we have $\hat A =M-N$, $T= M^{-1}N$, $c=M^{-1}\hat b$, and
the original linear system is equivalent to
\begin{equation} \label{equiv.sys}
A x = c, \text{ where } A =  M^{-1}\hat{A} = I-T, \enspace c = M^{-1}\hat{b}.
\end{equation}

For the convergence results on asynchronous Richardson iterations to come, we will always assume that the following assumptions are met:
\begin{align}
 & \mbox{$T$ is non-negative, i.e.\ $T \geq 0$ where $\geq$ is to be understood entrywise,} \label{eq:T_nonneg}\\
 & \mbox{$T$ is convergent, i.e.\ $\rho(T) < 1$,} \label{eq:T_convergent}\\
 & \mbox{$\spec(A) \subset \mathR^+$.} \label{eq:spec_in_Rplus} 
\end{align}
In other words, we are assuming that $ \hat A=M-N$ is a convergent weak splitting in the sense of \cite{Marek.Szyld.90}\footnote{See also
\cite{wozni:94,cliper:98a}.}
with the additional property that the spectrum of $T$ is real. Note that if $\hat{A}$ is symmetric and 
positive definite (spd), and $M$ is the diagonal or a block diagonal of $\hat{A}$, which then is
also spd, \eqref{eq:spec_in_Rplus} is fulfilled. If, in addition, $\hat A$ is a Stieltjes matrix, i.e.\ an M-matrix which is spd, 
and if again $M$ is the diagonal or a block diagonal of $\hat A$, then \eqref{eq:T_nonneg} and \eqref{eq:T_convergent} are also fulfilled; see 
\cite[Chapter 5]{berman-plemmons},
\cite[Section 3.5]{varga-book}, \cite[Chapter 11]{young}.


With the splitting $A = I-T$, the standard, synchronous iterative method is as follows.  Given $x^0$, for $k=0,1,\ldots$, compute
\begin{equation} \label{simple.it}
x^{k+1} = T x^k + c .
\end{equation}

We note then that 
if we denote $\lmin$ and $\lmax$ to be the smallest and largest eigenvalue of $A$, we have
\[
\lmin = 1-\rho, \quad \lmax \leq 1+\rho.
\]


\section{First order Richardson}

The first order Richardson method 
consists of taking a linear combination of the previous
iterate with that which would come from the standard iteration (\ref{simple.it}).
This method can be seen as the simplest case of a semi-iterative method \cite{EN, ENV, varga-book, young}. The 
sum of the coefficients of the linear combination must add up to one, since otherwise the method will not produce 
iterates that converge towards $A^{-1}b$.\footnote{Gene Golub in his thesis
\cite{Golub.thesis} calls this a method of averaging, following the nomenclature used by von Neumann.}

We first consider the stationary case where the parameter $\alpha$ defining the Richardson iteration
is fixed for all iterations.
We consider later the non-stationary case where $\alpha = \alpha_k$ depends on the iteration number.

The synchronous stationary iteration is
\begin{equation} \label{Richardson_1:eq}
x^{k+1} = (1 - \alpha) x^k + \alpha (T x^k +   c) =  x^k + \alpha [ c - (I-T) x^k] = x^k + \alpha r^k,
\end{equation}
where $r^k = c - (I-T) x^k$ is the residual of the equivalent system (\ref{equiv.sys}).

The convergence analysis of this synchronous method is straight-forward and well-known; see \cite[Section 11.4]{young}.
The analysis consists of analyzing the spectral
radius of the iteration matrix 
\[
T_{\alpha} = (1- \alpha) I + \alpha T = I - \alpha(I-T)= I - \alpha A.
\]
If $\mu \in \spec(T_{\alpha})$, then $\mu = 1-\alpha + \alpha \lambda$, with $\lambda \in \spec(T)$,
i.e., $\lambda \in [-\rho, \rho]$.

\begin{theorem} Let $\spec(A) \subset \mathR^+$.  Then
\begin{itemize}
  \item[(i)] iteration \eqref{Richardson_1:eq} converges if $\alpha \in (0,\;2/\lmax)$,
  \item[(ii)] the optimal choice is $\alpha = 2/(\lmin+\lmax)$ in the sense that this choice minimizes $\rho(T_\alpha)$,
  \item[(iii)] the optimal choice w.r.t.\ the information $\spec(A) \subset [a,b]$, $a > 0$ is $\alpha = 2/(a+b)$. 
\end{itemize}
\end{theorem}
\begin{proof}
We have $\spec(T_\alpha) = \{ 1 - \alpha \lambda: \lambda \in \spec(A)\}$ and thus  
\[
\rho(T_\alpha) = \max \{|1-\alpha \lmin|, \; |1-\alpha \lmax|\}.
\]
From this we see that $\rho(T_\alpha) < 1 $ iff $\alpha \in (0,\; 2/\lmax)$, which is (i), and that $\rho(T_\alpha)$ is minimal
if $1-\alpha \lmin = -(1-\alpha \lmax)$ which gives (ii). Part (iii) follows from equating  $1-\alpha a$ with $-(1-\alpha b)$.
\end{proof}

Note that in our situation we know $\spec(A) \subset [1-\rho, \, 1+\rho]$, and, by (iii) the optimal $\alpha$ w.r.t.\ this information is $\alpha = 1$.

For the asynchronous iteration, we analyze when 
$\rho(|T_\alpha|) < 1$.
We adopt the notation $w>0$ for $w \in \mathR^n$ if $w_i > 0$ for $i=1,\ldots,n$. Our analysis relies on the following often-used fact from 
non-negative matrix theory which we restate with its proof for convenience.
\begin{lemma} \label{non_negative:lem} Let $T \in \mathR^{n \times n}$, $T \geq 0$ with spectral radius $\rho$. Then for every $\varepsilon > 0$ there exists a positive vector $w_\varepsilon >0$, $w_\varepsilon \in \mathR^n$, such that 
\[
Tw_\varepsilon \leq (\rho+\varepsilon)w_\varepsilon.
\]
\end{lemma}
\begin{proof} For $\delta > 0$, let
\[
T_\delta = T + \delta E, \enspace \mbox{ where } E = \begin{bmatrix} 1 & \cdots & 1 \\ \vdots & \ddots & \vdots \\ 1 & \cdots & 1 \end{bmatrix}.
\]
Then $T_\delta$ has only positive entries, and by the Perron-Frobenius Theorem, 
see \cite{berman-plemmons, varga-book}, e.g., there exists $w_\delta > 0$ such that $T_\delta w_\delta
= \rho(T_\delta)w_\delta$ which, since $Ew_\delta \geq w_\delta$, gives
\begin{equation} \label{T_delta:eq}
Tw_\delta \leq (\rho(T_\delta)-\delta) w_\delta.
\end{equation}
By continuity of the spectral radius, we can choose $\delta = \delta(\varepsilon)$ such that $\rho(T_{\delta(\varepsilon)}) \leq \rho +\varepsilon$,
so that \eqref{T_delta:eq} becomes the assertion of the lemma (with $w_\varepsilon = w_{\delta(\varepsilon)}$).
\end{proof}

In Theorem~\ref{1st_order_asynchronous_convergence:thm} below, as well as in
Theorem~\ref{2nd_order_asynchronous_convergence:thm}, we will also use the fact that if, for $T \in \mathR^{n \times n}$, $T \geq 0$ and $w \in \mathR^n$, $w>0$, we have $Tw \leq \nu w$, then $\rho(T) \leq \nu$. This follows immediately from observing that $Tw \leq \nu w$ is equivalent to 
$\|T\|_w \leq \nu$ where $\| \cdot \|_w$ is the matrix norm induced by the weighted maximum norm $\|x\|_w = \max_{i=1}^n |x_i/w_i|$ on~$\mathR^n$.

\begin{theorem} \label{1st_order_asynchronous_convergence:thm} Assume that \eqref{eq:T_nonneg}, \eqref{eq:T_convergent} and \eqref{eq:spec_in_Rplus} hold 
and let $\rho = \rho(T)$. Then $\rho(|T_\alpha|) < 1$ if $\alpha \in (0,\tfrac{2}{1+\rho})$, where $\tfrac{2}{1+\rho} > 1$.
\end{theorem}
\begin{proof}
Let $\varepsilon > 0$ and, by Lemma~\ref{non_negative:lem}, let $w_\varepsilon > 0$ be a vector for which $Tw_\varepsilon \leq (\rho + \varepsilon)w_\varepsilon$. Then we have 
\[
| T_\alpha|w_\varepsilon 
~\leq~
 |1-\alpha|w_\varepsilon + \alpha Tw_\varepsilon \leq (|1-\alpha| + \alpha (\rho+\varepsilon))w_\varepsilon = \nu w_\varepsilon \mbox{ with } \nu = |1-\alpha| + \alpha (\rho + \varepsilon).
\]
For $0 < \alpha \leq 1$ we have $0 \leq \nu = (1-\alpha)+(\rho + \varepsilon) \alpha = 1-\alpha(1-(\rho+\varepsilon))$ which is less than 1 if $\varepsilon> 0$ is taken small enough. For $1 < \alpha < \tfrac{2}{1+\rho}$ 
we have $0 < \nu = (\alpha-1) + (\rho+\varepsilon)\alpha = (1+\rho+\varepsilon)\alpha - 1$ which, for $\alpha$ fixed, is again less than 1 for $\varepsilon$ sufficiently small.
\end{proof}

We note that $\alpha = 1$, the optimal parameter value for the synchronous iteration w.r.t\ the information
$\spec(A) \subseteq [1-\rho,1+\rho]$, is covered by this theorem.

We discuss now the case in which $\alpha = \alpha_k$, i.e., the case, where the first order Richardson parameter
changes from one iteration to the next. 
As long as 
\linebreak
$0 < \alpha_k \leq \overline{\alpha} < \tfrac{2}{1+\rho}$, the ``non-stationary'' asynchronous method
converges as well, using \cite[Corollary 3.2]{Frommer.Szyld.00}. 
In fact, using the latter result, we have the following 
\linebreak
theorem.

\begin{theorem}
Let $T_k: \mathC^n \to \mathC^n$, $k \in \mathbb{N}$  be a pool of linear operators sharing the same fixed point $x^* =A^{-1}b$ and 
being all contractive w.r.t.\ this fixed point in the same weighted max-norm, i.e., $\| T_k - x^*\|_w \leq \gamma_k \|x-x^*\|$ 
for all $x \in \mathC^n$. 
If $0 \leq \gamma_k \leq \gamma < 1$ for some 
$\gamma \in [0,1)$,  
then the asynchronous iterations which at each step picks one of the operators form the pool as 
its iteration operator, produces iterates which converge to $x^*$.  
\end{theorem}

The result for non-stationary first order Richardson follows by using the vectors $w_\varepsilon$ from Lemma~\ref{non_negative:lem} for $T$ and by observing that with $T_k = (1-\alpha_k)I + \alpha_k T$ we have 
\[
\|(1-\alpha_k)I + \alpha_k T \|_{w_\varepsilon} \leq |1-\alpha_k| + \alpha_k(\rho + \varepsilon)  \leq |1-\overline{\alpha}|+ \overline{\alpha} (\rho+\varepsilon).
\]
Taking $\varepsilon > 0$ such that $\rho + \varepsilon < 1$ and $(1+\rho+\epsilon) \overline{\alpha}-1 < 1$  
gives $|1-\overline{\alpha}|+ \overline{\alpha} (\rho+\varepsilon) < 1$.

\section{Second order Richardson}
The second order Richardson method is the semi-iterative method one obtains
when correcting $x^k$ with a linear combination of $(x^k - x^{k-1})$
and the residual at step $k$, rather than just the residual as used
in the standard iteration (\ref{simple.it}).
Equivalently, one can take $x^{k+1}$ to be a linear combination
of the first order Richardson iterate (\ref{Richardson_1:eq}) with just $x^{k-1}$, as follows,
\begin{eqnarray}
x^{k+1} &=& (1+\beta) [ (1 - \alpha) x^k + \alpha (T x^k +   c) ] - \beta x^{k-1 \nonumber } \\
&=& - \beta x^{k-1} + (1+ \beta) x^k + (1+\beta) \alpha [ - x^k + T x^k +   c)]   \nonumber \\
&=& x^k - \beta (x^{k-1} - x_k) + (1+\beta) \alpha [ c - (I-T) x^k)]  \nonumber \\
 &=& x^k + \beta(x^k-x^{k-1}) + (1+\beta) \alpha (c-Ax^k) \label{Richardson2:eq} \\
        & = & (1+\beta)(I-\alpha A)x^k -\beta x^{k-1} +(1+\beta)\alpha c, \; k=1,2,\ldots. \nonumber
\end{eqnarray}
In addition to $x^0$, it is now necessary to also prescribe $x^1$, and for this it is possible to use one step of 
(\ref{simple.it}) or one step of first order Richardson \cite{Golub.thesis}.

The results to come are more restrictive
than those for first order Richardson, since we can show 
the convergence of asynchronous second order 
Richardson only for parameter values which are quite far from the optimal ones for the synchronous iteration.


We can write the three-term recurrence in \eqref{Richardson2:eq} using a matrix of doubled size as follows, cf.\ \cite{Young.72}, 
\[
\begin{bmatrix} x^{k+1} \\ x^{k} \end{bmatrix}  = \underbrace{\begin{bmatrix} (1+\beta)(I - \alpha A) & -\beta I \\ I & 0 \end{bmatrix}}_{:=T_{\alpha,\beta}} 
     \begin{bmatrix} x^k \\ x^{k-1} \end{bmatrix} + \begin{bmatrix} (1+\beta)\alpha c \\ 0 \end{bmatrix}  \cdot
\]

For the synchronous iteration (\ref{Richardson2:eq}), two approaches 
have been used to analyze convergence.
For the first approach~\cite{Young.72}, we note that if $\lambda$ is an eigenvalue of $T_{\alpha,\beta}$ with eigenvector $\left[ \begin{smallmatrix} s \\ t \end{smallmatrix} \right]$,
then $s = \lambda t$ and $(1+\beta) [(I - \alpha A)] s - \beta t = \lambda s$, that is,
$ (1+\beta) (I - \alpha A) \lambda t - \beta t = \lambda^2 t$. Thus, assuming that $t \neq 0$, this implies that
$\det [ (1+\beta) (I - \alpha A)\lambda - \beta I - \lambda^2I ] =0$, so that for $\mu \in \spec(A)$, the eigenvalues of  $T_{\alpha,\beta}$
must satisfy the quadratic equation
\begin{equation} \label{poly_sync:eq}
\lambda^2 - (1+\beta) (1 - \alpha \mu)\lambda + \beta =0.
\end{equation}
Figure~\ref{fig:1} (first column) plots the spectral radius of 
$T_{\alpha,\beta}$ as a function of $\alpha$ and $\beta$ for three examples.


Frankel \cite{Frankel.50} shows that the values of the parameters $\alpha$ and $\beta$ that minimize the maximum of the modulus 
of the solution of \eqref{poly_sync:eq}
are given by
$\alpha = 2/(a+b)$ and $\beta = \left( \frac {\sqrt{b}-\sqrt{a}} {\sqrt{b}+\sqrt{a}} \right)^2 := q^2$, for $A$ such that
$\spec(A) \subset [a,b]$ with $a > 0$. The resulting minimal value for
$\rho(T_{\alpha,\beta})$, the spectral radius of the iteration operator, is $q$.

For the second approach \cite{Golub.thesis,Golub.Varga.61.partI},
assuming one uses the above optimal parameters,
the recurrence of the polynomials defining (\ref{Richardson2:eq})
is used to bound the 2-norm of the error as
\begin{equation} \label{asympt:eq}
\| x^k - x^* \|_2 \leq \left[ q^k \left( 1+ k \frac{1-q^2}{1+q^2} \right) \right] \|x^0 -x^*\|_2,
\end{equation}
where $x^*$ is the solution of (\ref{equiv.sys}).
Here, it is assumed that the first iterate is $x^1 = x^0 +\alpha(b-Ax^0)$.


In summary, 
the following is thus known for the synchronous iteration.
\begin{theorem} \label{2nd_order_synchronous_convergence:thm}  
Let $A$ be spd. Then
\begin{itemize}
\item[(i)] the optimal parameters w.r.t.\ the information $\spec(A) \subset [a,b]$ with $a > 0$ are $\alpha = 2/(a+b)$ and 
$\beta = \left( \tfrac{b-a}{a +b+2\sqrt{ab}}\right)^2 = \left( \frac{\sqrt{b}-\sqrt{a}} {\sqrt{b}+\sqrt{a}} \right)^2$, and the asymptotic convergence 
factor $\rho(T_{\alpha,\beta})$ is equal to $q =\frac {\sqrt{b}-\sqrt{a}} {\sqrt{b}+\sqrt{a}}$,
\item[(ii)] with these parameters and with $x^1 = x^0+\alpha(b-Ax^0)$, a bound for the 2-norm of the errors
is given in (\ref{asympt:eq}).
\end{itemize}
\end{theorem}

For the asynchronous second order Richardson, the following theorem proves convergence for certain ranges for $\alpha$ and $\beta$.

\begin{theorem} \label{2nd_order_asynchronous_convergence:thm}
Assume that \eqref{eq:T_nonneg}, \eqref{eq:T_convergent} and \eqref{eq:spec_in_Rplus} are fulfilled and let $\rho = \rho(T)$. Then we
have $\rho(|T_{\alpha,\beta}|) < 1$, provided 
\begin{equation} \label{2nd_order_asynchronous_coditons:eq}
\alpha > 0 \text{ and }  |1+\beta|(|1-\alpha| + \alpha \rho)+|\beta| < 1.
\end{equation}
\end{theorem}
Before we prove the theorem, consider the choice $\alpha = 1$. For this choice, the theorem states that asynchronous iterations converge for
$ -1 \leq \beta < \tfrac{1-\rho}{1+\rho}$, as can be seen from considering the two cases $\beta \geq 0$ and $-1 < \beta < 0$ separately. 
If the information about the spectral interval is $\spec(A) \subset [1-\rho,\,1+\rho]$,
Theorem~\ref{2nd_order_synchronous_convergence:thm} gives that
the optimal $\alpha$ for the synchronous iteration
is $\alpha = 1$, and the corresponding optimal 
$\beta$ will be close to $1$ for $\rho$ close to 1. The range of $\beta$ for which Theorem~\ref{2nd_order_asynchronous_convergence:thm} guarantees convergence of the asynchronous iteration for $\alpha = 1$, however, has $1-\rho$ as an upper bound for $\beta$ according to \eqref{2nd_order_asynchronous_coditons:eq}, and this will be close to 0 if $\rho$ is close to 1.
\par\noindent
{\em Proof of Theorem~\ref{2nd_order_asynchronous_convergence:thm}.} 
Let $\varepsilon > 0$ be small enough such that we still have
\[
 |1+\beta|(|1-\alpha| + \alpha (\rho+\varepsilon))+|\beta| < 1,
\]
and let $w_\varepsilon > 0$ be a vector with $Tw_\varepsilon \leq (\rho+\varepsilon)w_\varepsilon$ which exists by Lemma~\ref{non_negative:lem}.
Let $\gamma > 1$ and consider the vector $[\begin{smallmatrix} w_\varepsilon \\ \gamma w_\varepsilon \end{smallmatrix} ]$. 
Then, if $\alpha > 0$, we have
\begin{eqnarray*}
|T_{\alpha,\beta}| \begin{bmatrix} w_\varepsilon \\ \gamma w_\varepsilon \end{bmatrix} &=& \begin{bmatrix} |1+\beta| \cdot |I - \alpha A| & |\beta| I \\ I & 0 \end{bmatrix}
    \begin{bmatrix} w_\varepsilon \\ \gamma w_\varepsilon \end{bmatrix} \\
    &\leq& \begin{bmatrix} (|1+\beta| \cdot(|1-\alpha| + \alpha (\rho+\varepsilon)) + |\beta| \gamma)w_\varepsilon \\ w_\varepsilon \end{bmatrix}
    \,
    \leq \, \sigma_\varepsilon \begin{bmatrix} w_\varepsilon \\ \gamma w_\varepsilon \end{bmatrix},
\end{eqnarray*}
with
\begin{equation} \label{sigma:eq}
\sigma_\varepsilon = \max\{ \tfrac{1}{\gamma}, \; |1+\beta| \cdot (|1-\alpha| + \alpha (\rho+\varepsilon)) + |\beta| \gamma\}.
\end{equation}
Now, since $|1+\beta|(|1-\alpha| + \alpha (\rho+\varepsilon))+|\beta| < 1$, we can choose $\gamma > 1$ close enough to 1 such that we also have $|1+\beta|(|1-\alpha| + \alpha (\rho+\varepsilon))+\gamma |\beta| < 1$,
which gives $\sigma_\varepsilon < 1$
\linebreak  
in \eqref{sigma:eq}. 
\hspace*{4.5in} 
\cvd

We note that for $\beta < -1$, the inequality $|1+\beta|(|1-\alpha| + \alpha \rho)+|\beta| < 1$ cannot be fulfilled. 
Denoting $\nu := |1-\alpha| + \alpha \rho$ we can distinguish the two cases $0 \leq \nu < 1$ and $\nu \geq 1$. In the first case, we obtain
that  $|1+\beta|\nu+|\beta| < 1$ if $-1 \leq \beta < \tfrac{1-\nu}{1+\nu}$. In the second case, there is no $\beta$ which satisfies the inequality.

To compare with \eqref{poly_sync:eq}, let us study the eigenvalues of $|T_{\alpha,\beta}|$. 
We follow the same development as before for $T_{\alpha,\beta}$ and 
write:
$$
|T_{\alpha,\beta}| ~ \left[ \begin{array}{c} s \\ t \end{array} \right]  = 
\lambda  \left[ \begin{array}{c} s \\ t \end{array} \right] ~.
$$
Looking at the second block row of $|T_{\alpha,\beta}|$, we see that
$s = \lambda t$. Then, the first block row reads
$$
(|1+\beta| |I - \alpha A| \lambda + |\beta| I - \lambda^2 I) t =0.
$$
This means that 
$$
\det(|1+\beta| |I - \alpha A| \lambda + |\beta| I - \lambda^2 I) =0.
$$
For every eigenvalue $\mu = \mu_i$ of $|I - \alpha A|$ we thus have that $\lambda$ satisfies the quadratic equation
\begin{equation} \label{poly_async:eq}
\lambda^2 - |1+\beta|  \mu  \lambda - |\beta|  =0.
\end{equation}
Figure~\ref{fig:1} (second column) plots the spectral radius of
$|T_{\alpha,\beta}|$ as a function of $\alpha$ and $\beta$ for three examples.

\section{Discussion}

\begin{figure}[p]
\includegraphics[width = 0.49\textwidth]{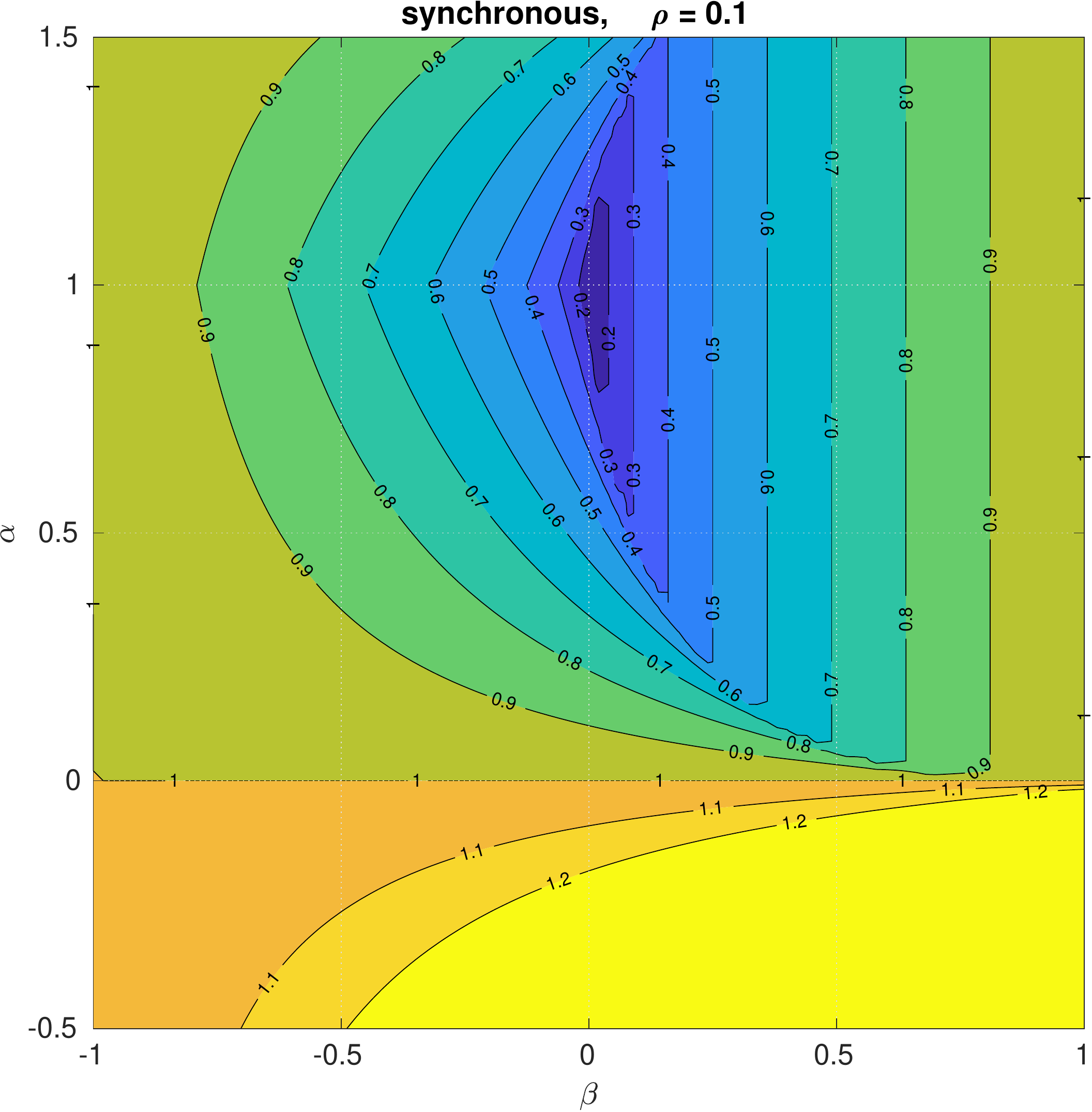} \hfill
\includegraphics[width = 0.49\textwidth]{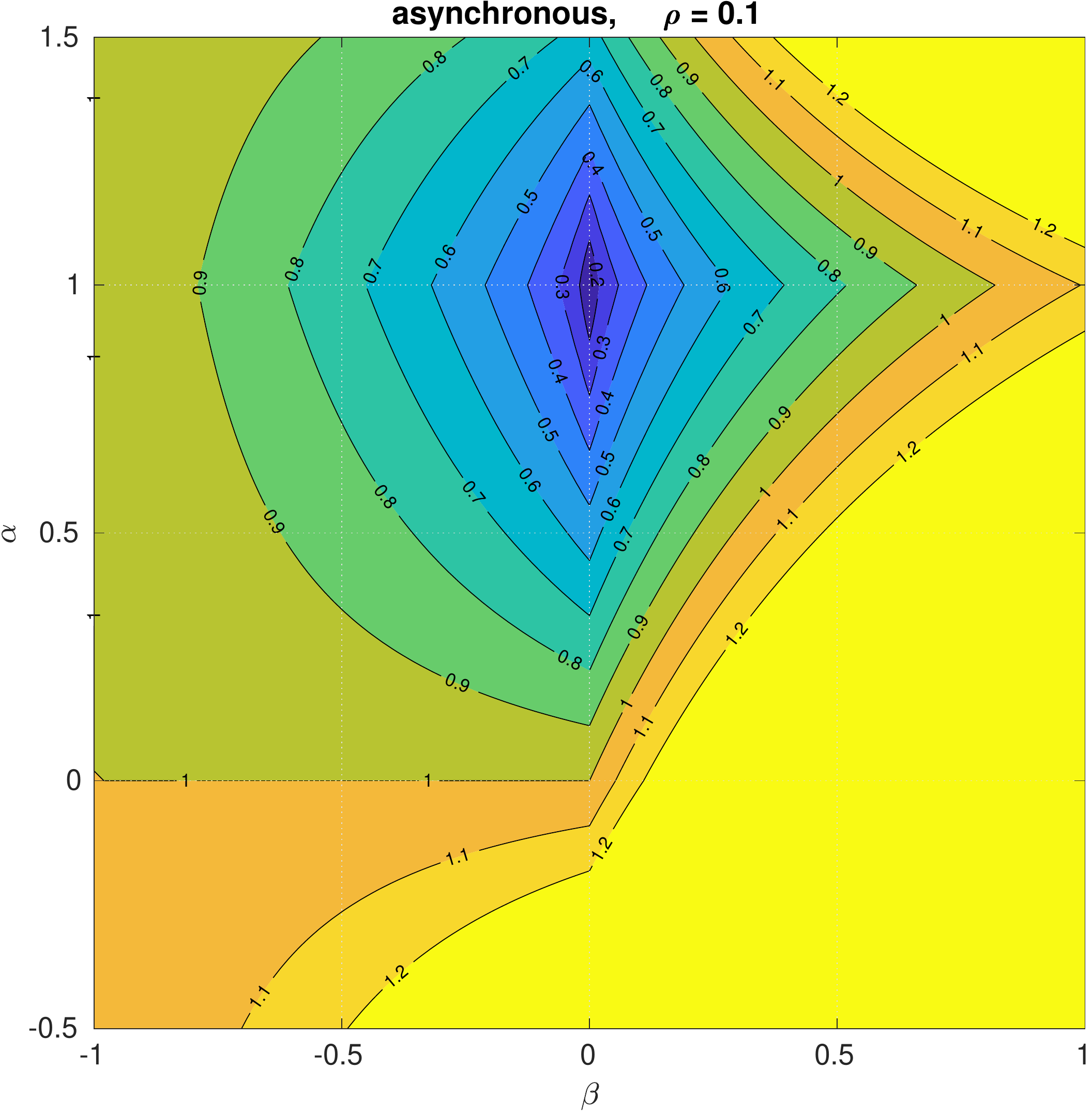} \\
\includegraphics[width = 0.49\textwidth]{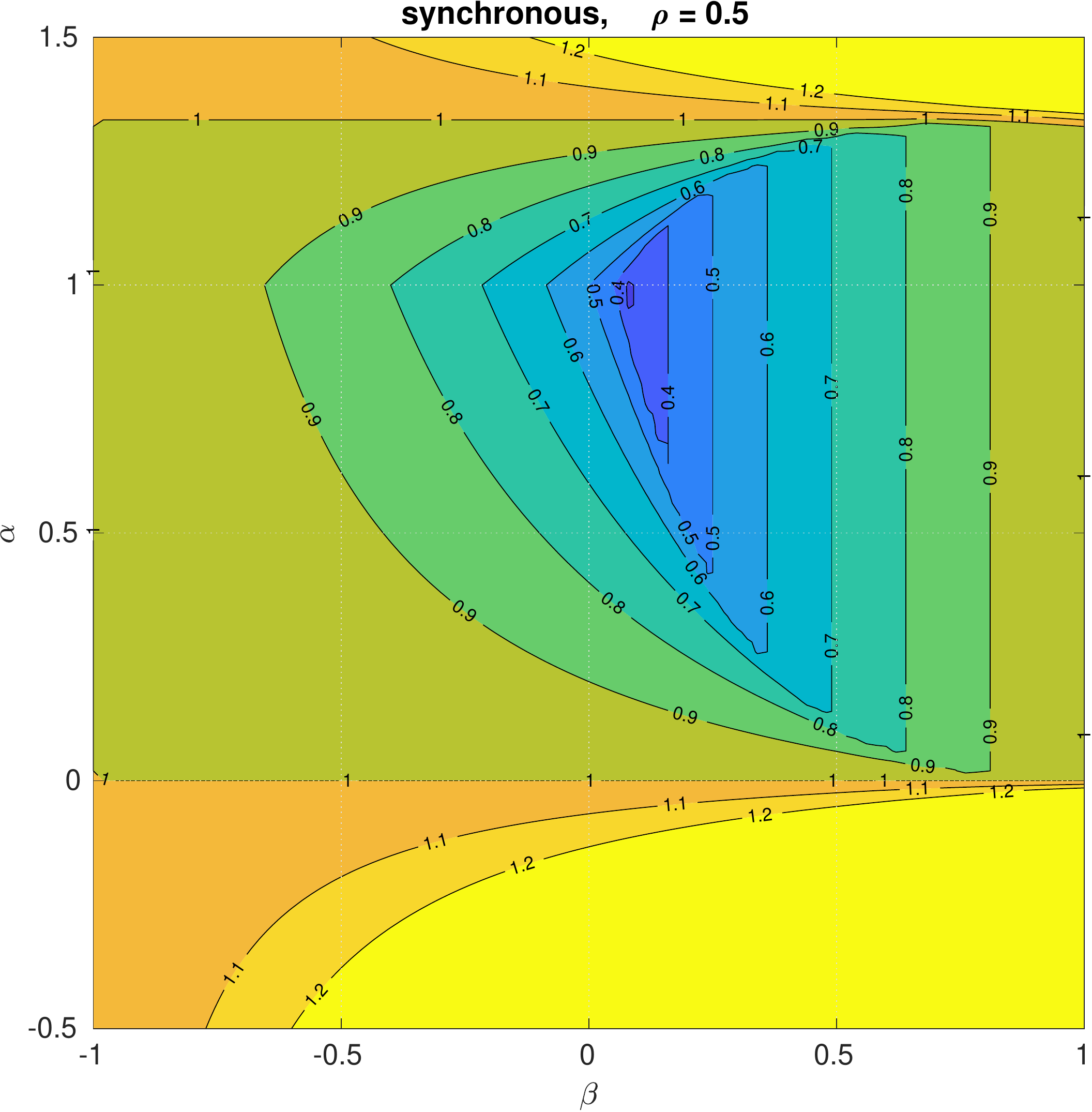} \hfill
\includegraphics[width = 0.49\textwidth]{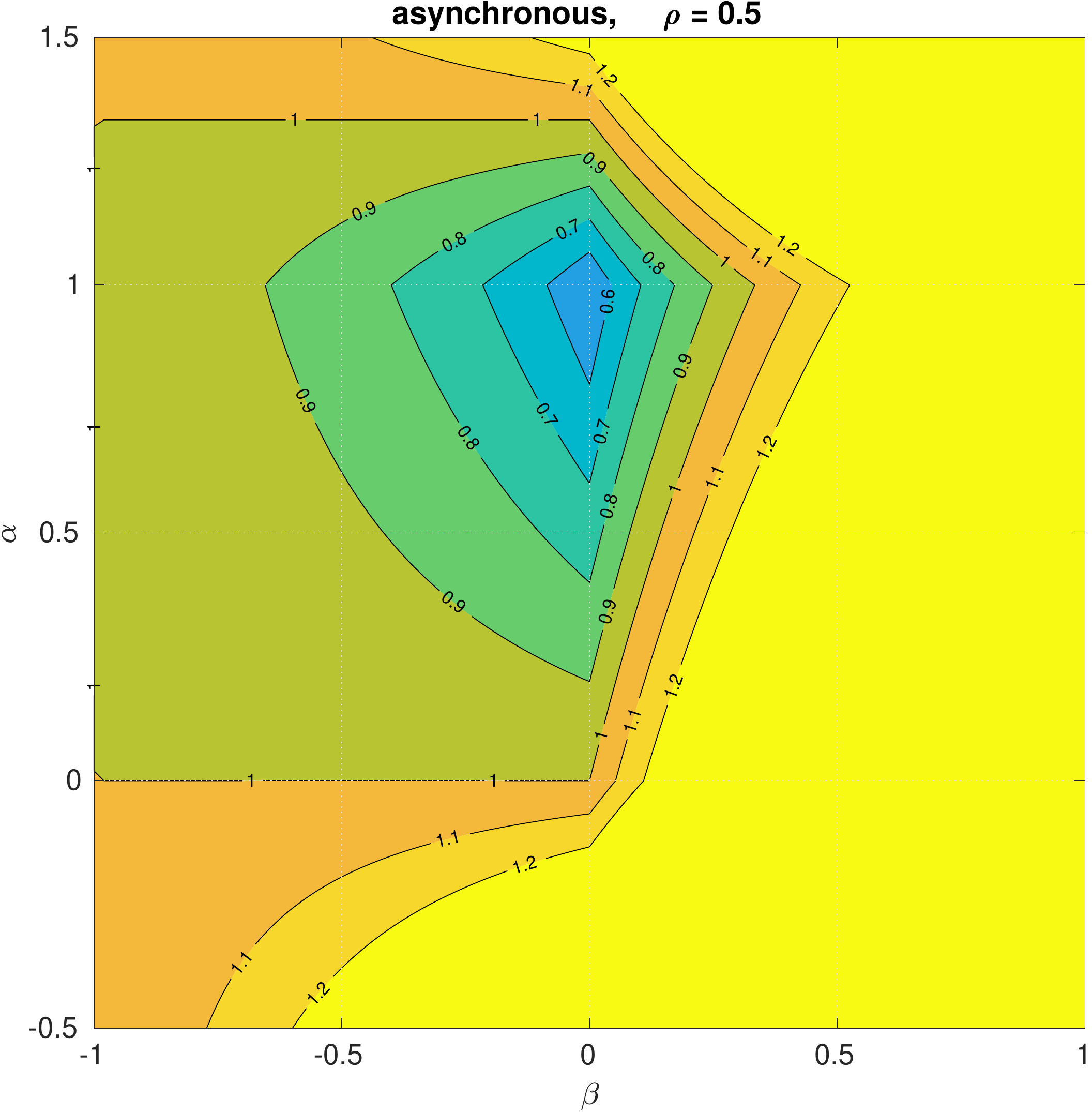} \\
\includegraphics[width = 0.49\textwidth]{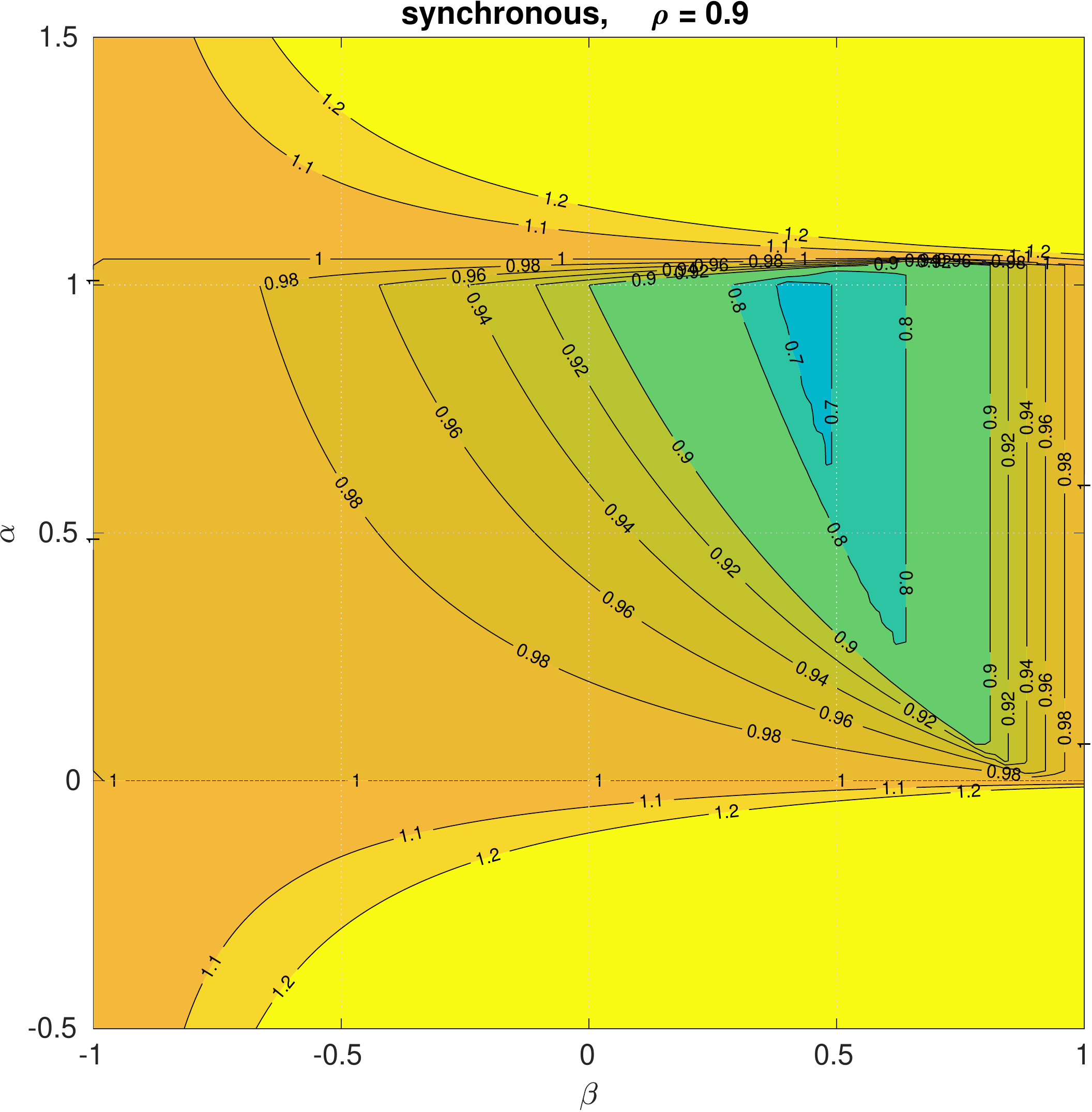} \hfill
\includegraphics[width = 0.49\textwidth]{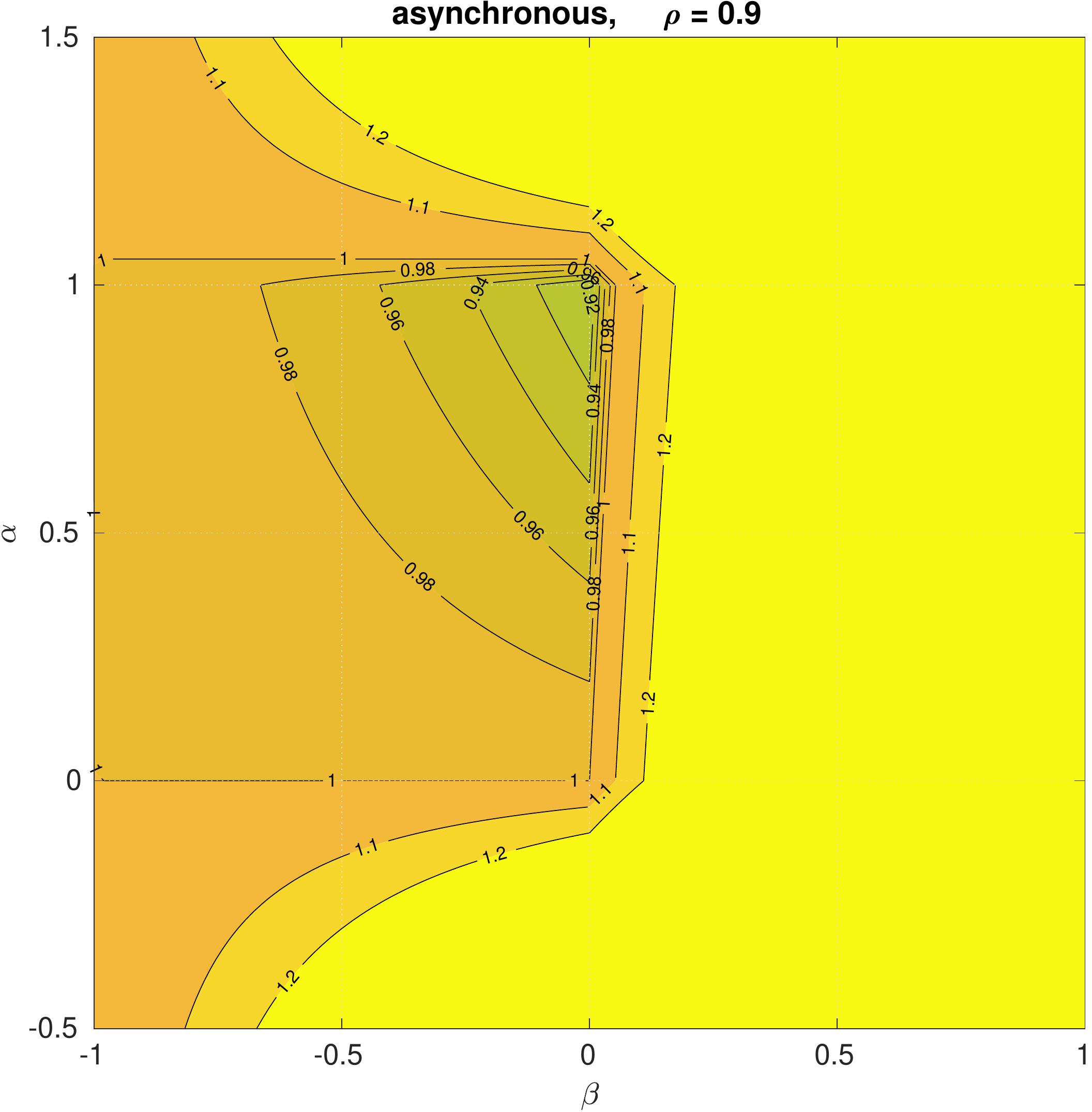}
\caption{\label{fig:1}
Spectral radius of $T_{\alpha,\beta}$ (synchronous case) and
of $|T_{\alpha,\beta}|$ (asynchronous case) 
as a function of $\alpha$ and $\beta$ when 
$\lambda_{\min}(A) = 1-\rho$ and $\lambda_{\max}(A) = 1+\rho$,
for three values of $\rho$.}
\end{figure}

For the second order Richardson method,
Figure~\ref{fig:1} plots the contours of the spectral radius of
$T_{\alpha,\beta}$ (synchronous case) and
of $|T_{\alpha,\beta}|$ (asynchronous case)
as a function of $\alpha$ and $\beta$ when
$\lambda_{\min}(A) = 1-\rho$ and $\lambda_{\max}(A) = 1+\rho$,
for $\rho$ equal to 0.1, 0.5, and 0.9.
The spectral radii were computed from the roots of the 
polynomials \eqref{poly_sync:eq} and \eqref{poly_async:eq}.
In our setting, the optimal $\alpha$ is always 1.

In the synchronous case, as $\rho$ increases, 
the optimal value of $\beta$ increases from near 0 toward 1.

The plots for the asynchronous case are best explained
in terms of the plots for the synchronous case.
When $\beta \le 0$, $\rho(|T_{\alpha,\beta}|)$
and $\rho(T_{\alpha,\beta})$ appear to be the same.
When $\beta > 0$, it appears that
$\rho(|T_{\alpha,\beta}|) > \rho(T_{\alpha,\beta})$.
In particular, the region where the spectral radius
is less than 1 is smaller in the asynchronous case
than in the synchronous case.
The effect is that  $\rho(|T_{\alpha,\beta}|)$ is smallest for $\beta = 0$, 
which corresponds
to the first order method.

Consider $\rho = 0.5$.  For the synchronous
case, the optimal $\beta$ is approximately 0.0718.
Although the asynchronous method can converge for this
value of $\beta$, the value of 0 gives a lower value of $\rho(|T_{\alpha,\beta}|)$.
Now consider $\rho = 0.9$.  For the synchronous
case, the optimal $\beta$, {minimizing $\rho(T_{\alpha,\beta})$}, is approximately 0.3929.
{Regarding the} asynchronous method {we have that $|T_{\alpha,\beta}|$} 
has spectral radius greater than 1
for this value of $\beta$.  To guarantee convergence,
the asynchronous method must use a very small value
of $\beta$.

These results are quite negative for the asynchronous
second order method.  However, in practice, the situation
could be more favorable.  The condition $\rho(|T_{\alpha,\beta}|) < 1$
for the asynchronous method guarantees that the method
will converge for any initial vector and any specific
asynchronous iterations, i.e., any choice of
the delays, $k-s_j(k)$, and any choice of the sets $J_k$ of components to update
(satisfying the mild conditions \eqref{eq:asy_delays} and \eqref{eq:asy_updates}).  In practice, the asynchronous
method may converge despite $\rho(|T_{\alpha,\beta}|) > 1$.
One could imagine that the ``degree
of asynchrony'' affects the convergence of the asynchronous
method, and we explore this next with numerical experiments.

\section{Numerical behavior}

The asynchronous first and second order Richardson methods were implemented
in parallel using multithreading and shared memory.  
Tests were run on a dual processor Intel Xeon computer
with a total of 20 cores.  The threads were pinned to
the cores using ``scatter'' thread affinity.

The test matrix $A$ arises from the standard finite difference Laplacian matrix $\hat A$
on a $100 \times 100$ grid of unknowns. With Jacobi preconditioning, 
the preconditioned matrix
$A$ remains spd and thus satisfies \eqref{eq:spec_in_Rplus}, while $T=I-A$ is the iteration matrix which satisfies \eqref{eq:T_nonneg} and \eqref{eq:T_convergent}.
A right-hand side vector was chosen randomly with components
chosen independently from the uniform distribution on
$(-0.5, 0.5)$.  The same vector was used for all tests.
The initial vector for all iterations was zero.

Different numbers of threads were used in different tests.  Each thread was
assigned approximately the same number of unknowns to update.
The iterations performed by each thread were terminated when 
the all the unknowns were updated
an average of 500 times.  Because the threads operate asynchronously,
the number of updates performed on each unknown is generally different.
We refer to the difference between the largest number of updates and the 
smallest number of updates as the {\em range}.
When the iterations are terminated, we measure the residual norm
relative to the initial residual norm.  The residual norm
is not calculated during the iterations, as such calculations
involving dot products induce synchronization in the method.

\subsection{First order Richardson}

For the asynchronous first order Richardson method,
Table~\ref{tab:0} shows the convergence results for tests with different numbers of
threads.  For the given matrix, the optimal $\alpha$ is 1.
For each number of threads, the method was run 100
times.  Columns 2 and 3 of the table show the average range, and the average relative
residual norm when the asynchronous iterations were terminated.
For comparison, the relative residual norm attained after 500 iterations
of the synchronous first order Richardson method is $1.691939\times 10^{-2}$.
Evidently, the convergence of the asynchronous method
is {\em better} than the convergence of the synchronous method.
This perhaps nonintuitive result is due to the fact that the asynchronous method
has a multiplicative effect \cite{jordi-ipdps18,wolfsonpou-chow-jpdc-2018}, 
i.e., unknowns are not all updated
at the same time, and when unknowns are updated, they are 
immediately available to other threads.  Indeed, for a single thread,
the asynchronous method corresponds to Gauss-Seidel, giving
a relative residual norm of $7.421009\times 10^{-3}$ which is lower than that
of the synchronous method, which corresponds to the Jacobi method.
As the number of threads is increased, convergence generally worsens slightly
as the method departs from a pure Gauss-Seidel method.
The convergence is always better than the convergence of the synchronous method
for all numbers of threads tested.

\begin{table}[htb]
\caption{Asynchronous first order Richardson for
different numbers of threads.
For comparison, the synchronous method attains an average relative residual
norm of $1.691939\times 10^{-2}$ for all numbers of threads.
Timings for the asynchronous and synchronous methods
for performing a fixed number of iterations are also given.}
\label{tab:0}
\centering
\footnotesize
\begin{tabular}{ccccc}
\hline
number of & average & average rel.&  async    & sync     \\
threads   & range   & resid. norm &  time (s) & time (s) \\
\hline
 1  &    0.0  &  $7.421009\times 10^{-03}$   &  0.060177  &  0.048345 \\
 2  &   17.1  &  $7.491060\times 10^{-03}$   &  0.034049  &  0.030291 \\
 3  &   76.1  &  $7.686441\times 10^{-03}$   &  0.022664  &  0.020642 \\
 4  &   98.3  &  $7.624358\times 10^{-03}$   &  0.018009  &  0.017360 \\
 5  &  129.6  &  $7.940683\times 10^{-03}$   &  0.015023  &  0.015171 \\
 6  &  138.1  &  $7.902309\times 10^{-03}$   &  0.012898  &  0.012751 \\
 7  &  144.6  &  $8.021550\times 10^{-03}$   &  0.011334  &  0.012374 \\
 8  &  172.2  &  $8.149458\times 10^{-03}$   &  0.010997  &  0.012067 \\
 9  &  240.4  &  $8.500669\times 10^{-03}$   &  0.010039  &  0.010737 \\
10  &  191.4  &  $8.248697\times 10^{-03}$   &  0.009339  &  0.010642 \\
11  &  222.4  &  $8.363452\times 10^{-03}$   &  0.009225  &  0.010741 \\
12  &  215.5  &  $8.311822\times 10^{-03}$   &  0.008861  &  0.010590 \\
13  &  248.9  &  $8.450671\times 10^{-03}$   &  0.009132  &  0.010339 \\
14  &  227.7  &  $8.416794\times 10^{-03}$   &  0.007867  &  0.009669 \\
15  &  253.7  &  $8.403988\times 10^{-03}$   &  0.009014  &  0.009998 \\
16  &  292.2  &  $8.610365\times 10^{-03}$   &  0.008414  &  0.009871 \\
17  &  284.6  &  $8.530868\times 10^{-03}$   &  0.008179  &  0.009668 \\
18  &  305.9  &  $8.573682\times 10^{-03}$   &  0.007307  &  0.009660 \\
19  &  288.4  &  $8.445288\times 10^{-03}$   &  0.007020  &  0.009496 \\
20  &  297.3  &  $8.448706\times 10^{-03}$   &  0.007200  &  0.009249 \\
\hline
\end{tabular}
\end{table}

The table also shows timings for the asynchronous method
and the synchronous method for different numbers of threads
(for performing a fixed number of iterations).
For small numbers of threads, the synchronous method is
faster in performing 500 iterations than the asynchronous
method in performing an average of 500 iterations by 
each thread.  This can be explained by two factors:
(1) the asynchronous method has more work to do because
each thread, after each iteration, needs to count how many 
iterations have been performed by other threads in order to 
decide whether to terminate, and (2) the asynchronous
method has more write invalidations of cache lines compared
to the synchronous method which
writes new values of $x$ to a separate array.
However, for large numbers of threads, despite these
two factors, the asynchronous method is faster (in performing 500 iterations),
due to the elimination of thread synchronization.  The overhead
of threads waiting for other threads in the synchronous
method is evidently larger when more threads are used.

\subsection{Second order Richardson}
\label{sec:result2}

\begin{table}[tbp]
\caption{Asynchronous second order Richardson for
different numbers of threads.
The parameter values $\alpha=1$ and $\beta \approx 0.93968$ that were used
are optimal for {\em synchronous} iterations.
For comparison, the synchronous method attains an average relative residual
norm of $1.258388\times 10^{-7}$ for all numbers of threads.
Timings for the asynchronous and synchronous methods are also given.}
\label{tab:1}
\centering
\footnotesize
\begin{tabular}{cccccc}
\hline
number of & average & average rel.& number of &  async    & sync     \\
threads   & range   & resid. norm & failures  &  time (s) & time (s) \\
\hline
 1  &    0.0  &  $1.258388\times 10^{-7}$  &   0  &  0.053275  &  0.052961 \\
 2  &   40.8  &  $4.235170\times 10^{-7}$  &   0  &  0.031146  &  0.032542 \\
 3  &  104.3  &  $6.175605\times 10^{-6}$  &   0  &  0.019592  &  0.023368 \\
 4  &  115.7  &  $1.444428\times 10^{-5}$  &   0  &  0.016493  &  0.018801 \\
 5  &  166.0  &  $1.495107\times 10^{-4}$  &   0  &  0.013533  &  0.017519 \\
 6  &  163.0  &  $4.524130\times 10^{-4}$  &   0  &  0.011563  &  0.014606 \\
 7  &  200.1  &  $1.868556\times 10^{-3}$  &   0  &  0.010649  &  0.013078 \\
 8  &  151.5  &  $9.259216\times 10^{-3}$  &   0  &  0.009794  &  0.012843 \\
 9  &  246.0  &  $4.035731\times 10^{-2}$  &   1  &  0.008917  &  0.012560 \\
10  &  203.2  &  $1.088207\times 10^{-1}$  &   1  &  0.009000  &  0.012371 \\
11  &  209.4  &  $4.582844\times 10^{-1}$  &  21  &  0.008972  &  0.011905 \\
12  &  185.5  &  $1.678645\times 10^{+0}$  &  25  &  0.008397  &  0.011527 \\
13  &  227.6  &  $1.046313\times 10^{+1}$  &  32  &  0.008216  &  0.011698 \\
14  &  205.9  &  $3.971405\times 10^{+1}$  &  43  &  0.007081  &  0.010863 \\
15  &  239.3  &  $5.207066\times 10^{+2}$  &  35  &  0.007568  &  0.010828 \\
16  &  166.8  &  $2.317140\times 10^{+2}$  &  24  &  0.007101  &  0.011470 \\
17  &  226.3  &  $3.303636\times 10^{+1}$  &  22  &  0.006217  &  0.011161 \\
18  &  191.8  &  $6.415417\times 10^{+1}$  &  30  &  0.005972  &  0.010969 \\
19  &  237.6  &  $2.377968\times 10^{+1}$  &  23  &  0.006237  &  0.011147 \\
20  &  173.8  &  $3.136173\times 10^{+1}$  &  46  &  0.006614  &  0.011012 \\
\hline
\end{tabular}
\end{table}

\begin{table}[tbp]
\caption{Asynchronous second order Richardson for
different numbers of threads.
Parameter values: $\alpha=1$ and $\beta=0.9$.}
\label{tab:2}
\centering
\footnotesize
\begin{tabular}{ccccc}
\hline
number of & average & average rel.& number of &  time \\
threads   & range   & resid. norm & failures  & (sec.)\\
\hline
 1  &    0.0  &  $9.566179\times 10^{-5}$   &  0  &  0.053059 \\
 2  &   47.7  &  $1.032052\times 10^{-4}$   &  0  &  0.030998 \\
 3  &  105.8  &  $1.802432\times 10^{-4}$   &  0  &  0.019752 \\
 4  &  122.3  &  $1.499666\times 10^{-4}$   &  0  &  0.016426 \\
 5  &  148.3  &  $2.081259\times 10^{-4}$   &  0  &  0.013676 \\
 6  &  154.7  &  $2.091337\times 10^{-4}$   &  0  &  0.011510 \\
 7  &  208.8  &  $2.745261\times 10^{-4}$   &  0  &  0.010352 \\
 8  &  182.9  &  $2.802124\times 10^{-4}$   &  0  &  0.010104 \\
 9  &  230.9  &  $3.434991\times 10^{-4}$   &  0  &  0.009003 \\
10  &  190.7  &  $2.701899\times 10^{-4}$   &  0  &  0.008824 \\
11  &  185.7  &  $3.500390\times 10^{-4}$   &  0  &  0.008086 \\
12  &  154.8  &  $3.445788\times 10^{-4}$   &  0  &  0.008059 \\
13  &  198.9  &  $6.526787\times 10^{-4}$   &  0  &  0.008342 \\
14  &  219.4  &  $2.479312\times 10^{-3}$   &  0  &  0.007052 \\
15  &  212.1  &  $8.821667\times 10^{-3}$   &  0  &  0.008112 \\
16  &  158.8  &  $2.594421\times 10^{-3}$   &  0  &  0.006902 \\
17  &  227.1  &  $1.113219\times 10^{-3}$   &  0  &  0.006715 \\
18  &  191.0  &  $6.389028\times 10^{-3}$   &  0  &  0.006050 \\
19  &  227.5  &  $1.464582\times 10^{-3}$   &  0  &  0.006365 \\
20  &  173.2  &  $4.955854\times 10^{-3}$   &  0  &  0.006487 \\
\hline
\end{tabular}
\end{table}

For the asynchronous second order Richardson method,
Table~\ref{tab:1} shows the convergence results for different numbers of
threads using the values $\alpha=1$ and $\beta \approx 0.93968$ which are
optimal for the synchronous method.  For these values, the asynchronous
method is not guaranteed to converge.
For each number of threads, the method was run 100
times.  The table shows the average range, the average relative
residual norm, and the number of failures, which is the number
of times the relative residual norm is greater than unity in the
100 runs.  

When a single thread is used, the asynchronous method is 
mathematically identical to the synchronous method.
When a small number of threads was used, the asynchronous method
always converged in the 100 runs, with a degradation in the
``convergence rate'' as the number of threads is increased.
What we mean here with convergence rate is how small is the residual
when the termination criterion is satisfied.
When a larger number of threads was used, the number of 
failures of the asynchronous method generally increases.  This is
due to an increased degree of asynchrony, which is somewhat
reflected by the increasing average range.

The table also shows timings for the asynchronous and 
synchronous second order Richardson methods.
The asynchronous method is faster (when performing a fixed
number of iterations) when more than 1 thread
is used, and the difference is generally larger when
more threads are used.

To attempt to make the asynchronous method more robust,
we test using a smaller value of $\beta$.
This is analogous to underestimating the bounds of the spectrum in the inexact
Chebyshev method \cite{golub-overton-1988}.
Table~\ref{tab:2} shows the convergence results using
$\alpha=1$ and $\beta = 0.9$.  With this value of $\beta$,
the asynchronous method is still not guaranteed to converge,
but it can be observed that convergence is always obtained
in the 100 runs for each number of threads.
However, the convergence rate is degraded for this choice
of $\beta$, i.e., compared to Table~\ref{tab:1} when 
a small number of threads is used.

\subsection{Synchronous and asynchronous convergence timings}

In the previous subsections, we compared the timings of asynchronous and 
synchronous iterations for a fixed number of iterations.  In this subsection,
we compare the residual norms that are achieved in parallel implementations
of the synchronous and asynchronous methods as a function of time.

In these tests, we used ten threads on a single Intel Xeon processor
with 10 cores, with each thread pinned to one of the two hyperthreads on
each core.  The test matrix is again from the standard finite difference
Laplacian matrix, but now on a $300 \times 300$ grid of unknowns.

The $90,000$ unknowns were partitioned into 10 partitions and each thread
was assigned to update the unknowns in one partition.  Two types of
partitionings were used: {\em balanced}, where each partition contains
9000 unknowns; and {\em unbalanced}, where 5 partitions contain 6000
unknowns and 5 partitions contain 12000 unknowns.

Figure \ref{fig:1st_timings} shows the results for the first order
Richardson method (using the optimal $\alpha=1$).
This figure was generated by running the parallel method for a
fixed number of iterations, $t$, in the synchronous case, or
when all threads have executed an average of $t$ iterations
in the asynchronous case.  The relative residual norm was then
computed.  For a given value of $t$, 20 tests were performed
and the average relative residual norm was computed.
These averages for different values of $t$ are plotted
in the figure, where the $x$-axis is the average execution time 
for tests with a given $t$.  The variations in residual norms 
and timings for a given $t$ are very small, and practically
indiscernible from the averages if they were plotted.

Figure \ref{fig:1st_timings} shows that, for first order Richardson for
the given test system, the asynchronous method is faster than the synchronous
method.  The unbalanced case is significantly slower than the balanced case
for the synchronous iteration, whereas for the asynchronous iteration there
is only a minor difference between the balanced and the unbalanced case.

\begin{figure}[h]
\centering
\parbox{0.45\textwidth}{
 \includegraphics[width = .45\textwidth]{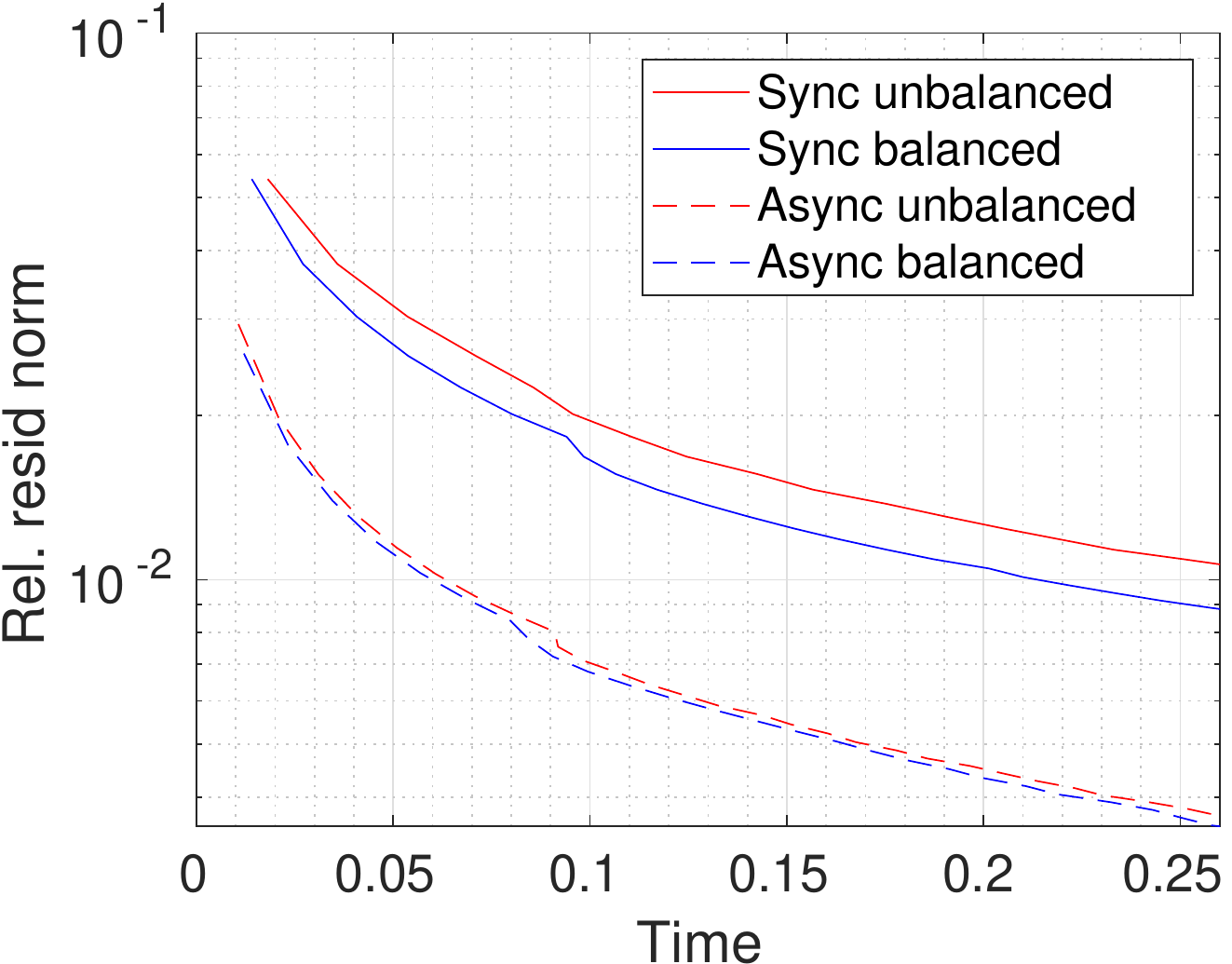}
 \caption{First order Richardson convergence with time.}
 \label{fig:1st_timings}}
\qquad
\parbox{0.45\textwidth}{
 \includegraphics[width = .45\textwidth]{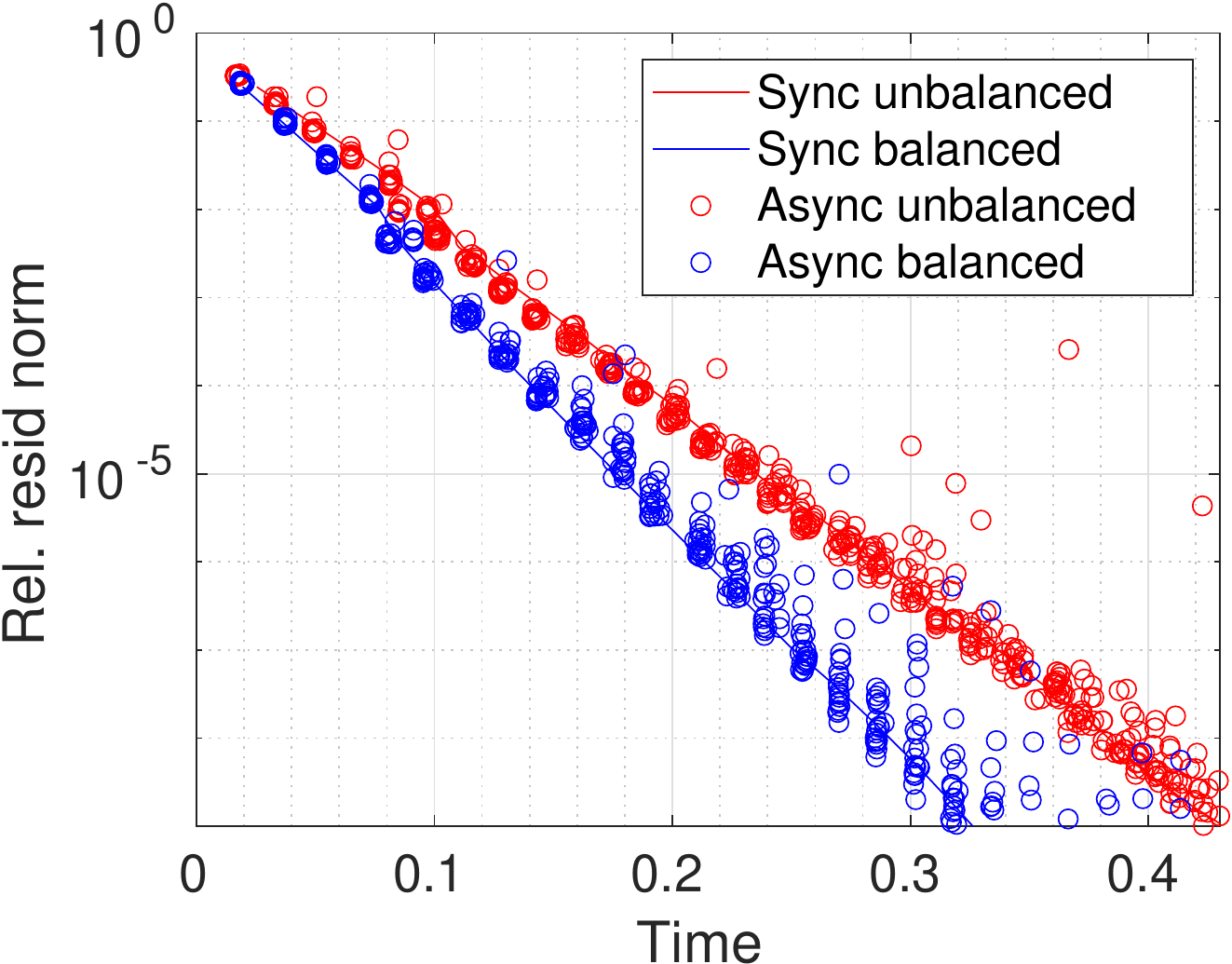}
 \caption{Second order Richardson convergence with time.}
 \label{fig:2nd_timings}}
\end{figure}

Figure \ref{fig:2nd_timings} shows the results for the second order
Richardson method, using the optimal values of $\alpha$ and $\beta$.
This figure was generated in the same way as the previous figure,
but here, the result of each of the 20 tests is plotted individually in the asynchronous
case, since the variations in the results are now much larger.
For the synchronous iteration (solid lines), the unbalanced case
is slower than the balanced case, as expected.
The asynchronous iteration (circles) is sometimes faster and
sometimes slower than the synchronous iteration (circles above and below
the solid lines of the same color).  We also observe that
the asynchronous method is substantially slower when the partitions
are unbalanced, compared to when they are balanced.
This is in contrast to the observation for first order Richardson,
which was not as sensitive to imbalance.
As we had observed in Section \ref{sec:result2}, an increased
degree of asynchrony is detrimental to the convergence of the
second order Richardson method, and here it is the load imbalance 
that increases the degree of asynchrony.

Comparing the asynchronous first and second order Richardson methods
for the given test problem with the use of optimal parameter values, the
second order method can converge much faster than the first
order method.  Convergence can be reliable although
it is not guaranteed.

\section{Conclusion}

Except to say whether or not an asynchronous iterative method will
converge in the asymptotic limit, the convergence behavior of these
methods is strongly problem-dependent and computer platform-dependent
and not well covered by theory.
For the first and second order Richardson methods, in the setting where
$\rho(T) < 1$, $T \ge 0$, and $\spec(A) \subset \mathR^+$, this paper provides a
description of the parameter values for which the asynchronous versions
of these methods are guaranteed to converge.  Numerically, however, we
find that this theoretical description can give a pessimistic view of
asynchronous iterative methods.  For a standard test problem, a multithreaded
parallel implementation of asynchronous iterations can converge reliably
in cases where it is theoretically possible for such iterations
to diverge.  How likely divergence will occur depends on the degree
of asynchrony in the computation, which is difficult to quantify.
A possible theoretical approach is to analyze asynchronous
iterative methods as randomized algorithms \cite{avron-jacm-2015}.

Asynchronous execution of the first order Richardson method
can clearly give much lower time-to-solution than synchronous execution.
Asynchronous execution of the second order Richardson method
may be slightly faster than synchronous execution because each iteration 
performed by a thread is executed more rapidly.  On the other
hand, execution may be slower because asynchrony is detrimental to
the convergence of the method.
The second order method, with its use of not one but two previous iterates,
appears to require much tighter coupling between the threads that
are working in parallel.

\section*{Acknowledgments} Work on this paper commenced while the three
authors were attending a workshop at the Centre International de
Rencontres Math\'ematiques, Luminy, France in September 2019.
The center's support for such an event is greatly appreciated.
Work of the first and third authors was supported in part by
the U.S.\ Department of Energy under grants DE-SC-0016564 and DE-SC-0016578.

\bibliographystyle{plain}
\bibliography{Richardson}

\end{document}